\documentclass[12pt,leqno,twoside]{amsart}

\usepackage{mathrsfs, amsmath, amsthm, amsfonts, amssymb, latexsym, amscd, graphicx,tikz-cd}

\usepackage{fullpage}

\newcommand{\R}{\mathbb{R}}

\newtheorem{thm}{Theorem}[section]
\newtheorem{cor}{Corollary}[section]

\newtheorem*{defn}{Definition}

\begin{document}
\title{Combinatorial Hodge Index Theorem for Polytopes}
\author{Jacob B. Wood} 
\address{Department of Mathematics, University of Wisconsin-Madison, 480 Lincoln Drive, Madison WI 53706-1388, USA}
\email{jbwood2@wisc.edu}

\begin{abstract}

Toric varieties can be constructed from rational polytopes, and several invariants of toric varieties can be expressed in terms of the combinatorics of the corresponding polytope. Barthel--Brasselet--Fieseler--Kaup (BBFK) introduced combinatorial intersection cohomology for convex polytopes, which agrees with the intersection cohomology of the associated toric variety when the polytope is rational. Maxim--Sch\"urmann computed the intersection cohomology signature of a projective toric variety, corresponding to the case of a polytope with rational vertices. Using the combinatorial framework of BBFK, we show that the Maxim--Sch\"urmann formula extends to arbitrary convex polytopes. Finally, we discuss a version of the Hodge index theorem for polytopes.
    
\end{abstract}

\date{\today}

\subjclass[2020]{52B05, 52B20, 14C17, 14C30, 14M25}

\keywords{convex polytope, fan, toric variety, combinatorial intersection cohomology, K\"ahler package, Hard Lefschetz, Hodge-Riemann relations, signature, $h$-polynomial, $g$-polynomial}

\maketitle

\section{Introduction}

        We compute the signature of the combinatorial intersection cohomology pairing of a polytope. This generalizes a formula of Maxim-Sch\"urmann \cite[Example 5.8, Example 6.7]{MS} for the intersection cohomology signature of a projective toric variety to non-rational convex polytopes. When restricted to rational polytopes, this yields a combinatorial proof of the signature formula, without resorting to mixed Hodge modules, characteristic classes or the geometry of the associated toric variety.

        The study of the combinatorics of polytopes saw significant breakthroughs with the introduction of Stanley's toric $h$-polynomials. For a rational polytope, the coefficients of its toric $h$-polynomial are the intersection cohomology Betti numbers of the corresponding projective toric variety \cite{S}. Thus, the combinatorially defined $h$-polynomial has in this case a symmetric unimodal sequence of non-negative coefficients by the hard Lefschetz theorem for intersection cohomology of complex projective algebraic varieties \cite{BBD}. However, there exist polytopes for which no combinatorially equivalent polytope has all rational coordinates \cite{Z}, so such a polytope does not have an associated toric variety. Without a variety, one might ask whether there is another way to associate a sequence of vector spaces to any polytope. This was answered by Barthel-Brasselet-Fieseler-Kaup \cite{BBFK} by introducing a special sheaf on the poset of cones of the normal fan to the polytope. Their ``minimal extension sheaf'' yields a graded module that behaves much like the intersection cohomology of a toric variety; this graded module is called the combinatorial intersection cohomology module. Its graded pieces are $\mathbb{R}$-vector spaces, and the work of Barthel-Brasselet-Fieseler-Kaup \cite{BBFK}, Bressler-Lunts \cite{BL}, and Karu \cite{K} show that they satisfy combinatorial versions of the hard Lefschetz theorem and Hodge-Riemann relations.

        The hard Lefschetz theorem for simplicial polytopes was proven combinatorially by McMullen and subsequently extended to all polytopes by Karu, whose proof relied on choices to define an intersection pairing. This pairing was subsequently justified by Bressler-Lunts \cite{BL2} (see also \cite{BBFK3}), who were able to define a canonical pairing and show that it coincides with that of Karu. We use the primitive decomposition induced by the Lefschetz operator to combinatorially reproduce the signature formula for general polytopes. We then define an intersection cohomology  $\chi_y$-polynomial for a fan to interpret the signature formula as a version of the Hodge index theorem.

{\bf Acknowledgments.} The author would like to thank Lauren\c{t}iu Maxim for suggesting the problem. This paper was produced completely without the aid of any Large Language Model (LLM), including in the research process, preparation, and editing.

\section{Preliminaries}

Let $V$ be a finite-dimensional vector space over $\R.$

\begin{defn} A \textbf{cone} $\sigma$ is a subset of $V$ defined by non-negativity of a finite set of linear functionals on $V$.
\end{defn}

\begin{defn} A \textbf{face} $\tau\subset\sigma$ is a cone of $V$ defined by all the same inequalities as $\sigma$ along with another inequality $-L\geq0,$ where $L$ is a functional non-negative on $\sigma$.
\end{defn}

\begin{defn} A \textbf{fan} $\Phi$ is a finite set of cones in $V$ such that the intersection of any two cones is a face of each, and every face of a cone of $\Phi$ also belongs to $\Phi$. A codimension 1 face is called a \textbf{facet} and a one dimensional face is a called a \textbf{ray}.
\end{defn}

\begin{defn} {The \textbf{support} of a fan $\Phi$ is the union of its cones denoted $|\Phi|\subset V$. A fan whose support is all of $V$ is called \textbf{complete}}.
\end{defn}

Note that as a set, $\Phi$ is a finite poset with relations given by inclusion of subfans. This poset structure endows $\Phi$ with a topology generated by the principal order ideals of the poset. In other words, the open sets are the subfans of $\Phi$.

Stanley's $h$- and $g$-polynomials for an Eulerian poset, e.g., the face lattice of a polytope, are defined inductively as follows:
\[g_\emptyset(t)=h_\emptyset(t)\equiv1\]
\[g^j_Q=\begin{cases}
                h_Q^j-h_Q^{j-1} \text{ for } 0\leq j\leq \lfloor d/2\rfloor  \\
                0,  \text{ otherwise}
                 \end{cases}\]
\[h_P(t)=\sum_{Q<P}(t-1)^{d-(rk(Q)+1)}g_Q(t).\]

\begin{defn} A fan is \textbf{simplicial} if the dimension of each cone is equal to the number of its rays.
\end{defn}

Consider the sheaf of $\R$-algebras \[\mathcal{A}_\Phi\in Sh(\Phi)\] characterized by its stalks $\mathcal{A}_{\Phi,\sigma}=A_\sigma$, the algebra of polynomial functions on $\sigma$. Then the global sections $\mathcal{A}_\Phi(\Phi)=\Gamma(\Phi,\mathcal{A}_\Phi)$ can be interpreted as the set of piecewise polynomial functions on $\Phi$ whose restrictions to any cone of $\Phi$ are polynomial functions. Another important sheaf on $\Phi$ is the constant sheaf $A_\Phi$ whose global sections are the ring of global polynomials \[A:=\text{Sym}\ V^*=\Gamma(A_\Phi).\] Denote by $A^+$ the maximal ideal of polynomials of non-zero degree. For an $A$-module, $M$, we define \[\overline{M}=M/A^+M.\]

The fan analog of intersection cohomology is defined in terms of a special sheaf of $\mathcal{A}_\Phi$-modules. Bressler-Lunts \cite{BL} showed that there exists a (unique up to isomorphism) sheaf of $\mathcal{A}_\Phi$-modules called the \textbf{minimal extension sheaf}, denoted $\mathcal{L}_\Phi$, satisfying the following:
$\mathcal{L}_{\Phi,o}=\R,$ and the map $\overline{\mathcal{L}_{\Phi}(\sigma)}\rightarrow\overline{\mathcal{L}_\Phi(\partial\sigma)}$ is an isomorphism. For a toric variety $X_\Phi$, its equivariant (resp.  non-equivariant) intersection cohomology can be expressed as $IH_T(X_\Phi)\cong \Gamma(\mathcal{L}_{\Phi}) \ (\text{resp. } IH(X_\Phi)\cong\overline{\Gamma(\mathcal{L}_\Phi)})$. In light of these results, we can define the (combinatorial) intersection cohomology of a fan to be $$IH(\Phi):=\overline{\Gamma(\mathcal{L}_\Phi)}.$$

It was shown in \cite[Theorem 6.3]{BBFK} that combinatorial intersection cohomology admits a Poincare pairing. In general, the pairing is not explicit, but for simplicial fans (in which case $\mathcal{L}_\Phi=\mathcal{A}_\Phi$), Brion \cite{Bri} gave the following description.

Let $f,g\in \Gamma(\mathcal{L}_\Phi)=\Gamma(\mathcal{A}_\Phi)$ be continuous real-valued functions on $V$ which are homogeneous polynomials on each cone. Denote by $f_\sigma$ the global polynomial which agrees with $f$ on $\sigma$, and denote by $\phi_\sigma$ a product of a choice of defining functions on each ray of $\sigma$ such that the resulting product is positive on the interior of $\sigma$. (Note that this choice of defining function for each ray may not be consistent from cone to cone.) Then the pairing is given by $$(f,g)=\sum_{\dim\sigma=n}\frac{f_\sigma g_\sigma}{\phi_\sigma}.$$
\cite{BBFK} furthermore shows that even in the case of fans coming from nonrational polytopes, the pairing is a non-degenerate bilinear form on combinatorial intersection cohomology. In Theorem \ref{th33}, we give a combinatorial formula for the signature of this pairing. The signature of a non-degenerate $\R-$bilinear form on a vector space $W$ is the difference in dimension between the positive definite and negative definite subspaces $W^+,W^-\subseteq W.$

\section{The Combinatorial K\"ahler Package and signature of fans}

For toric varieties, there is a Lefschetz operator on intersection cohomology, and we would like to find such an operator for combinatorial intersection cohomology. It is unknown whether there is such an operator for general fans. However, when $\Phi$ admits a strictly convex piecewise linear function (e.g., the inner normal fan of a convex polytope $P$), then that function as an element of $\mathcal{A}_{\Phi}(\Phi)$ is a Lefschetz operator for $IH(\Phi)$. Fans which have a strictly convex piecewise linear function are called \textbf{projective} because when the fan has an associated toric variety, this condition is equivalent to the toric variety being projective (see \cite[6.9]{D}). When $\Phi$ is projective, Karu \cite[Theorem 0.1]{K} showed:

\begin{thm} (Hard Lefschetz) \label{th31} Let $\Phi$ be a projective fan with a strictly convex piecewise linear function $\ell$. Then the map \[\ell^k:IH^{n-k}(\Phi)\rightarrow IH^{n+k}(\Phi)\] is an isomorphism for all $k\geq 1.$
\end{thm}

Observe the following consequences of the hard Lefschetz theorem. First, the (intersection) Betti numbers of a fan $\Phi$, which we will denote \[ih^j_\Phi:=\dim IH^j(\Phi),\] form a $\mathbf{symmetric}$ sequence, so \[ih_\Phi^j=ih_\Phi^{2n-j}.\] Furthermore, since the composition of functions $\ell^{k}:IH^{n-k}(\Phi)\rightarrow IH^{n+k}(\Phi)$ is an isomorphism, the first iteration $\ell:IH^{n-j}(\Phi)\rightarrow IH^{n-j+2}(\Phi)$ is an injection, and the last map $\ell:IH^{n+j-2}(\Phi)\rightarrow IH^{n+j}(\Phi)$ is a surjection. Thus, the sequence of even index Betti numbers $ih_\Phi^{2j}$ is $\mathbf{unimodal}$. The kernel of the surjection $\ell^{j+1}:IH^{n-j}(\Phi)\rightarrow IH^{n+j+2}(\Phi)$ is called the $\mathbf{primitive\ subspace}$ of $IH^{n-j}(\Phi)$, which we denote \[\text{Prim}_\ell IH^{n-j}(\Phi)=\text{ker}(\ell^{j+1}|_{IH^{n-j}(\Phi)}).\] For the dimensions of the primitive subspaces, we write \[ip_\Phi^{n-j}:=\dim \text{Prim}_\ell IH^{n-j}(\Phi)=ih_\Phi^{n-j}-ih_\Phi^{n-j-2}.\] It is sometimes convenient to encode these invariants as the coefficients of a polynomial, so we define: $$ih_\Phi(q):=\sum_{j=0}^{2n} ih_\Phi^j q^j$$
$$ip_\Phi(q):=\sum_{j=0}^{n} ip_\Phi^j q^j.$$

The Hodge-Riemann relations for fans proven by Bressler-Lunts \cite{BL2} describe how the cohomology pairing, $(\cdot,\cdot)_\Phi:IH^{n-k}(\Phi)\times IH^{n+k}(\Phi)\rightarrow IH^{2n}(\Phi)$, behaves on each primitive summand.

\begin{thm} (Hodge-Riemann Relations) Let $\Phi$ and $\ell$ be as in Theorem \ref{th31}. Then for any non-zero $a\in \text{Prim}_\ell IH^{n-k}(\Phi)$ we have $$Q_\Phi(a):=(-1)^\frac{n-k}{2}(a,\ell^ka)_\Phi>0.$$
\end{thm}

Combining the above two theorems, we obtain the following result analogous to the proof of the classical Hodge index theorem \cite{H}. This gives a combinatorial proof of a similar formula obtained by Maxim-Sch\"urmann \cite[Example 5.8, Example 6.7]{MS} in the toric context (i.e., for rational polytopes). Note that in the notations of \cite{MS}, our result is stated for their polar polytope, see \cite[Formula (6.13)]{MS}, which explains the difference in powers of $(-2)$ in our formulae. 

\begin{thm} (Signature of a projective fan) \label{th33} Let $\Phi$ be the projective fan of dimension $n$ associated to a convex polytope. The signature of its intersection pairing is given by $$\sigma(\Phi)=\sum_{\Delta\prec\Phi} g_\Delta(-1)(-2)^{n-\dim\Delta-1}$$
where $g_\Delta(t)$ is Stanley's $g$-polynomial associated to the proper face $\Delta$.
\end{thm}

\begin{proof} We will begin by computing the signature $\sigma(\Phi):=\sigma(IH^n(\Phi))$ in terms of the Betti numbers of $\Phi$. Then we use the fact that the Betti numbers coincide with certain combinatorial invariants.

When $n$ is odd, the middle intersection cohomology group is $IH^n(\Phi)=0$ because all odd dimensional intersection cohomology vanishes, so $\sigma(\Phi)=0$. Observe that for odd $n$, hard Lefschetz (Theorem \ref{th31}) implies symmetry of the non-zero Betti numbers: $ih^{2k}=ih^{2n-2k}$. Thus, 
\begin{equation}\sum_{k=0}^n (-1)^kih_{\Phi}^{2k}=0=\sigma(\Phi).
\end{equation}

When $n$ is even, we obtain the following orthogonal decomposition of $IH^n(\Phi)$ by the hard Leschetz theorem:
\begin{equation}
   IH^n(\Phi)=\bigoplus_k \ell^{k}\text{Prim}_\ell IH^{n-2k}(\Phi), 
\end{equation}
so we have
\begin{equation}
\begin{split}
    \sigma(\Phi)&=\sum_{0\leq k\leq \frac{n}{2}}\sigma(\text{Prim}_\ell IH^{n-2k}(\Phi)) \\
   & \overset{(HR)}{=}\sum_{0\leq k\leq \frac{n}{2}}(-1)^\frac{n-2k}{2}\text{dim}(\text{Prim}_\ell IH^{n-2k}(\Phi)),\\
&=\sum_{0\leq k\leq \frac{n}{2}}(-1)^\frac{n-2k}{2}(ih_\Phi^{n-2k}-ih_\Phi^{n-2k-2}),
\end{split}
\end{equation}
where the sign of the terms in the second equality is given by the Hodge-Riemann relations. Furthermore, rearranging terms shows 

\begin{equation}\label{eq1}
    \sum_{0\leq k\leq \frac{n}{2}}(-1)^\frac{n-2k}{2}(ih_\Phi^{n-2k}-ih_\Phi^{n-2k-2})=\sum_{0\leq k\leq \frac{n}{2}}(-1)^\frac{n-2k}{2}ih_\Phi^{n-2k}+\sum_{0\leq k\leq \frac{n}{2}}(-1)^{\frac{n-2k}{2}+1}ih_\Phi^{n-2k-2}
\end{equation}
Then reindexing the second summand of the right side of Equation \ref{eq1}, we have that
\begin{equation*}
\begin{split}
    \sum_{0\leq k\leq \frac{n}{2}}(-1)^\frac{n-2k}{2}ih_\Phi^{n-2k}&+\sum_{0\leq k\leq \frac{n}{2}}(-1)^{\frac{n-2k}{2}+1}ih_\Phi^{n-2k-2}=2\left(\sum_{0\leq k< \frac{n}{2}}(-1)^\frac{n-2k}{2}ih_\Phi^{n-2k}\right)+(-1)^{\frac{n}{2}} ih_\Phi^n \\
&   \overset{(HL)}{=}\sum_{0\leq k< \frac{n}{2}}(-1)^\frac{n-2k}{2}ih_\Phi^{n-2k}+\sum_{0\leq k< \frac{n}{2}}(-1)^\frac{n+2k}{2}ih_\Phi^{n+2k}+(-1)^{\frac{n}{2}} ih_\Phi^n \\
&    =\sum_{0\leq r\leq n}(-1)^rih_\Phi^{2r},
\end{split}
\end{equation*}
where the equality labeled (HL) follows from the hard Lefschetz theorem, which implies the symmetry of Betti numbers.

From now on, we treat both the even and the odd cases together. Bressler-Lunts show that the Hard Lefschetz theorem implies that local $IH$ Betti numbers are well-defined \cite[Corollary 7.4]{BL} and that the following formulae \cite[Proposition 7.7]{BL} relate $ih_\Phi(q)$ to $ip_\Phi(q)$:
\[ip_o(q)=ih_o(q)\equiv1\]
\[ip^j_\sigma=\begin{cases}
                ih_\sigma^j-ih_\sigma^{j-2} \text{ for } 0\leq j\leq d  \\
                0 \text{ otherwise}
                 \end{cases}
\]
\[ih_\sigma(q)=\sum_{\Delta\in\sigma}(q^2-1)^{d-\text{dim}\Delta-1}ip_\Delta(q).\]
They then compare with Stanley's $h$- and $g$-polynomials whose relations coincide with those of the $ip$ and $ih$ polynomials of $\Phi$ when we declare the poset to be the poset of cones under inclusion and the poset's rank function to be the dimension of the cone minus 1. Thus, the pairs of polynomials are defined by the same inductive procedure with the same base case, so they must have the same coefficients, and $$h_\Phi(q^2)=ih_\Phi(q),\ \text{{and}}$$
$$g_\Phi(q^2)=ip_\Phi(q).$$
Then the signature $\sigma(\Phi)$ can be rewritten as
$$\sigma(\Phi)=\sum_i(-1)^iih_\Phi^{2i}=\sum_i(-1)^ih_\Phi^{i}=\sum_{\Delta\in\Phi} g_\Delta(-1)(-2)^{d-\text{dim}\Delta-1},$$
which completes the proof.
\end{proof}

\section{Towards a combinatorial Hodge index theorem for fans}

In the case of projective toric varieties, the intersection cohomology is of Hodge-Tate type, meaning that the Hodge structure of $IH^*(X_\Phi)$ is concentrated in bidegrees $(p,p)$, e.g., see \cite{Sa, MS}. In the absence of geometry, we define combinatorial intersection cohomology of fans to have the same property: 
\[
IH^{p,q}(\Phi):= 
\begin{cases} 
IH^{p+q}(\Phi), & p=q \\ 
0, & otherwise.
\end{cases}
\] 
Denote the $p$-th graded piece of $IH^j(\Phi)$ by \[Gr_F^p(IH^j(\Phi)):=IH^{p,j-p}(\Phi).\] Then we arrive at the following reinterpretation of the Poincar\'e polynomial of a fan in terms of the intersection cohomology Hodge $\chi_y$-polynomial:
\begin{equation*}
    \begin{split}
    \chi^{IH}_y(\Phi)&:=\sum_{p\leq j}(-1)^j\text{dim Gr}^p(IH^j(\Phi))(-y)^p \\
    & =\sum_{p=1}^{\dim \Phi}\text{dim Gr}^p(IH^{2p}(\Phi))(-y)^p \\
    & =\sum_{p=1}^{\dim \Phi}ih_\Phi^{2p}\cdot(-y)^p \\
    &=ih_\phi(t) \\
    &=h_\Phi(t^2).
    \end{split}
\end{equation*}
with $t^2=-y$.

Restricting to $y=1,$ we get the combinatorial analog of the Hodge index theorem: 
\begin{cor}
   With the above notations, we get for a projective fan $\Phi$ the following identity:
   $$\sigma(\Phi)=\chi^{IH}_1(\Phi).$$
\end{cor}

Finally, observe that when we return to the geometric situation, this is consistent with the Hodge index theorem 
$$\sigma(X)=\chi^{IH}_1(X)$$
for a (possibly singular) projective toric variety (proved for arbitrary projective varieties in \cite{MSS} and for compact varieties in \cite{BPS}).

\end{document}